\documentclass[english,12pt]{article} 
\pdfoutput=1
\usepackage{authblk}
\usepackage[top=2cm,bottom=2cm,left=1cm,right=1cm]{geometry}
\usepackage{fontenc}
\usepackage[utf8]{inputenc}
\usepackage{afterpage}
\usepackage[dvipsnames]{xcolor}
\usepackage{newtxtext,commath}
\usepackage{amscd}
\usepackage{amsmath}
\usepackage{amsthm}
\usepackage{newtxmath}
\usepackage{pdfsync}
\usepackage{hyperref}
\usepackage{amsfonts}
\usepackage{color}
\usepackage{enumerate}
\usepackage{graphicx}
\usepackage{amsfonts}
\usepackage{ragged2e}
\usepackage{lipsum}
\usepackage{bbm, fancyhdr}

\newtheorem{theorem}{Theorem}[section]
\newtheorem{corollary}[theorem]{Corollary}

\newtheorem{definition}[theorem]{Definition}
\newtheorem{remark}[theorem]{Remark}
\newtheorem{prop}[theorem]{Proposition}

\numberwithin{equation}{section}

\newcommand{\R}{\mathbb{R}}
\newcommand{\Sp}{\mathbb{S}^N_+}
\newcommand{\Hyp}{\mathbb{H}^N}

\newcommand{\be}{\beta}
\newcommand{\eps}{\varepsilon}

\newcommand{\prt}{\partial}

\newcommand{\dvg}{\text{div}}
\newcommand{\dvol}{dv_g}
\newcommand{\dist}{\text{dist}}
\newcommand{\cut}{\text{cut}}
\newcommand{\arcosh}{\text{arcosh}}
\newcommand{\camp}{\mathcal{X}(M)}
\newcommand{\V}{V_{p,a_1,...,a_n}}

\DeclareMathOperator{\hess}{Hess}

\title{Sharp multipolar $L^p$-Hardy-type inequalities on Riemannian manifolds} 
\date{}
\author[1]{Cristian Ciulică} 
\author[1,2]{Teodor Rugină}

\affil[1]{Faculty of Mathematics and Computer Science,	University of Bucharest, 14 Academiei Street, 010014 Bucharest, Romania.}
\affil[2]{Gheorghe Mihoc-Caius Iacob Institute of Mathematical
	Statistics and Applied Mathematics of the Romanian Academy\\
	050711 Bucharest, Romania.}
\affil[ ]{Emails: cristiciulica@yahoo.com, teorugina@yahoo.com}

\begin{document}
\maketitle \raggedright

\begin{abstract}
In this paper we prove sharp multipolar Hardy-type inequalities in the Riemannian $L^p-$setting for $p\geq 2$ using the method of super-solutions and fundamental results from comparison theory on manifolds, thus generalizing the work in \cite{cazacu-zuazua}, \cite{kristaly1} for $p=2$. We emphasize that when we restrict to Cartan-Hadamard manifolds, the inequalities improve in the case $2<p<N$ compared to the case $p=2$ since we obtain positive remainder terms which are controlled by curvature estimates. In the end, we treat the cases of positive and negative constant sectional curvature.    
\end{abstract}

\section{Introduction}
The Hardy inequality is a fundamental result in functional analysis and has undergone extensive study over the years due to its profound applications in mathematical physics, analysis and partial differential equations. This work aims to prove multipolar Hardy-type inequalities in the setting of Riemannian manifolds and their sharp formulations, which strongly depend on the geometric and topological properties of these manifolds.  \\
The classical Hardy inequality in Euclidean space $\R^N$, $N\geq 3$, states that for any $u\in C^\infty_c(\R^N)$:
\begin{equation}   \label{Hardy-clasic}
    \int_{\R^N} \abs{\nabla u}^p dx \geq \left( \frac{N-p}{p} \right)^p \int_{\R^N} \frac{\abs{u}^p}{\abs{x}^p} dx,
\end{equation}
where $1<p<N$ and the constant in the right-hand-side is sharp and not attained in the space $C^\infty_c(\R^N)$. Moreover, no positive remainder terms can be added in the inequality \eqref{Hardy-clasic}. Hardy  first proved in a discrete form in \cite{hardy1}, while continuous analogues and improvements were developed later in \cite{hardy2}. This inequality is classified in the literature as a unipolar inequality due to the singularity of $\frac{1}{\abs{x}^p}$ in $x=0$. Unipolar inequalities in $\R^N$ and various Hardy-type inequalities were intensively studied in the last decades in works such as \cite{barbatis}, \cite{brezis}, \cite{devyver}, \cite{marcus}, \cite{vazquez}.\\
In the context of Riemannian manifolds, the inequality \eqref{Hardy-clasic} is generalized using the Riemannian distance function to a point in the pioneering paper \cite{carron}. Later, the subject of Hardy inequalities on Riemannian manifolds and its different forms got a lot of attention, see for example the papers: \cite{kristaly1}, \cite{kombe1}, \cite{kombe2}, \cite{kristaly3} for functional inequalities and uncertainty priciples in general setting of Riemannian manifolds; \cite{berchio1} and \cite{flynn2} for results on Cartan-Hadamard manifolds; \cite{berchio2} and \cite{flynn1} for inequalities in the hyperbolic space; \cite{kristaly2} for a new technique to prove Hardy-type and many functional inequalities using the so-called Ricatti pairs, a generalisation of the concept of Bessel pairs introduced in \cite{ghoussoub}.\\
It is important to mention that when we pass from the flat case of $\R^N$ to Riemannian manifolds, the presence of curvature generates additional restrictions for the validity of such inequalities, as it is explained in \cite{carron}. Another effect of the curvature is that it can either strengthen or weaken the inequality, through the estimation of its terms. Thus, we may find, for example, that we obtain improved inequalities when the manifold is strongly negatively curved. We shall explore these aspects in the next sections.   \\
Multipolar Hardy inequalities involve a singular potential defined by a locally integrable function with multiple singular points. Such inequalities were studied in the last years in \cite{adimurthi}, \cite{bosi}, \cite{canale}, \cite{cazacu-zuazua}, \cite{felli}, \cite{hoffman} and references within. It is worth highlighting here the work in \cite{bosi} where this type of inequalities were obtained by the method of expansion of the square. Those results were improved later in \cite{cazacu-zuazua}, using the method of super-solutions to derive inequalities involving multipolar potentials of the form $W:= \sum_{1\leq i<j\leq n} \frac{\abs{a_i-a_j}^2}{\abs{x-a_i}^2\abs{x-a_j}^2}$, where $n\geq 2$ and $a_1,...,a_n$ are isolated points in $\R^N$, $N\geq 3$. Their result states that for any $u\in C^\infty_c(\R^N)$: 
\begin{equation}    \label{ec:cazacu-zuazua}
    \int_{\R^N} \abs{\nabla u}^2 dx \geq \frac{(N-2)^2}{n^2} \sum_{1\leq i<j\leq n} \int_{\R^N} \frac{\abs{a_i-a_j}^2}{\abs{x-a_i}^2\abs{x-a_j}^2} \abs{u}^2 dx.
\end{equation}
Moreover, they proved that the inequality is sharp, meaning that the constant $\frac{(N-2)^2}{n^2}$ cannot be improved. This result was later generalized in two directions, which we shall explore briefly in the following lines.
The first direction is the generalization to the $L^p$-setting of \eqref{ec:cazacu-zuazua} in \cite{cazacu-rugina} for the bipolar case (i.e. for two singular points) where the authors proved that for any $N\geq 3$ and $u\in C^\infty_c(\R^N)$: 
\begin{align}    \label{ec:cazacu-rugina}
    \int_{\R^N} \abs{\nabla u}^p dx & \geq \frac{p-1}{4}\left(\frac{N-p}{p-1}\right)^p \int_{\R^N} \frac{\abs{a_1-a_2}^2\abs{x-a}^{p-2}}{\abs{x-a_1}^p\abs{x-a_2}^p} \abs{u}^p dx    \notag\\
    & + \frac{p-2}{2}\Bigg( \frac{N-p}{p-1} \Bigg)^{p-1} \int_{\R^N} \frac{\abs{x-a}^{p-4}}{\abs{x-a_1}^p\abs{x-a_2}^p}\Big[\abs{x-a_1}^2\abs{x-a_2}^2 - \big( (x-a_1)\cdot(x-a_2) \big)^2\Big] \abs{u}^p dx.
\end{align}
Here $a$ stands for the median point of the segment $[a_1,a_2]$, precisely $a=\frac{a_1+a_2}{2}$. For $2\leq p<N$ it is shown that the constant $\frac{p-1}{4}\left(\frac{N-p}{p-1}\right)^p$ is sharp and not achieved in the energy space   $\mathcal{D}^{1,p}(\R^N)=\Big\{u\in\mathcal{D}'(\R^N) \;\Big|\; \int_{\R^N} \abs{\nabla u}^p dx < \infty \Big\}$. We notice that, as an improvement to \eqref{ec:cazacu-zuazua}, inequality \eqref{ec:cazacu-rugina} contains a positive remainder term when $p>2$ and a slightly different structure of the singular potential, which has a degeneracy in $a$ in the first integral when $p>2$ and in the second integral when $p>4$, while for $p\in (2,4)$ we have an extra singularity in $a$ in the second integral. \\
The second generalization of \eqref{ec:cazacu-zuazua} is its extension to Riemannian manifolds, where the curvature plays a significant role. In order to state the result in this curved setting, fix $N\geq 3$, $(M,g)$ a complete Riemannian manifold of dimension $N$, a set of points $\{a_1,...,a_n\}\in M$, $n\geq 2$ and the distance functions $d(x,a_i)=:d_i(x)$, for any $i=1,...,n$. We denote by $\Delta_g$ the Laplace-Beltrami operator on $M$, by $\nabla_g$ the gradient operator, and by $\dvol$ the canonical volume form on $M$. The authors in \cite{kristaly1} proved that for any $u\in C^\infty_c(M)$:
\begin{equation}    \label{ec:kristaly-p=2}
    \int_M \abs{\nabla_g u}^2 \dvol \geq \frac{(N-2)^2}{n^2} \sum_{1\leq i<j\leq n}\int_M \abs{\frac{\nabla_g d_i}{d_i}-\frac{\nabla_g d_j}{d_j}}^2 \abs{u}^2 \dvol + \frac{N-2}{n} \sum_{i=1}^n \int_M \frac{d_i\Delta_g d_i - (N-1)}{d_i^2} \abs{u}^2 \dvol,
\end{equation}
where the constant $\frac{(N-2)^2}{n^2}$ is sharp for $n=2$. It is interesting to remark that they also obtain a remainder term which depends on the curvature, as opposed to \eqref{ec:cazacu-zuazua}. This is due to the presence of $\Delta_g d_i$. This remainder is shown to be positive when $M$ is a Cartan-Hadamard manifold. Properties of Cartan-Hadamard manifolds, which shall be useful in our analysis, will be summarised in Section \ref{preliminaries}. \\
Our aim is to find a generalization of \eqref{ec:cazacu-rugina} to the context of Riemannian manifolds, in the spirit of the generalisation from \eqref{ec:cazacu-zuazua} to \eqref{ec:kristaly-p=2}. We start in Section \ref{preliminaries} by presenting some preliminaries on Riemannian geometry, listing some basic properties and two well-known comparison theorems. In Section \ref{main-results} we state the main results, while Section \ref{proof} will be dedicated to the proofs. We will begin by identifying the structure of singularities when we pass from $p=2$ to any $p$ with $1<p<N$ by employing an adaptation of the technique in \cite{picone} to Riemannian manifolds. Then we continue by analyzing the sharp constants and proposing minimizing sequences in the resulting inequalities, while also finding suitable conditions which guarantee the positivity of the remainder terms. In Section \ref{constant-curvature} we explore these inequalities in the constant curvature setting of the hyperbolic space $\Hyp$ and the upper-hemisphere $\Sp$.

\section{Preliminaries}   \label{preliminaries}
In this section, we recall some well-known concepts and definitions of differential operators on Riemannian manifolds. We also introduce the sectional curvature and the comparison theorems which are fundamental for the proof of the main results. For a more comprehensive understanding, one can check the books \cite{lee} or \cite{petersen}.
\begin{definition}
    Let $M$ be a $N$-differentiable manifold, $N\geq 3$. A Riemannian metric on $M$ is a $2-$tensor field $g$ that is symmetric and positive definite.  
\end{definition}
Thus, a Riemannian metric determines an inner product on the tangent space $T_xM$, for any $x\in M$. Let $TM=\cup_{x\in M} T_xM$ be the tangent bundle of $M$ and $\mathcal{X}(M)$ be the set of $C^\infty(M)$ vector fields on $M$ (i.e. sections in the tangent bundle $TM$).  We denote the norm of a tangent vector $X_x\in T_xM$ by $|X_x|:=g(X_x,X_x)^{\frac{1}{2}}$. If $(U,\varphi)$ is a chart on $M$, with $(x_i)$ local coordinates, we denote by $g_{ij}=g\left(\frac{\prt}{\prt x_i},\frac{\prt}{\prt x_j}\right)$ the local coefficients of $g$ in $U$. 
Recall that the distance function associated to the Riemannian metric $g$ is defined as a function $d: M \times M \to [0, \infty)$ given by
\begin{equation}     \label{ec:defdist}
    d(x,y) = \inf_{\gamma} \left\{ \text{lenght}(\gamma) \;|\; \gamma:[a,b] \rightarrow M,\; \text{piecewise differentiable},\; \gamma(a)=x,\; \gamma(b)=y \right\}.
\end{equation}
For $a\in M$ and $r>0$,  we denote by $B_r(a):=\{x\in M\ | d(x,a)<r\}$ the ball of radius $r$ and centered in $a$. Moreover, for $0<r<R$, let $A_a[r,R]:=\{x\in M \;\|\; r\leq d(x,a)\leq R \}$.  \\
Let $u:M\to \mathbb R$ be a function of class $C^1$ and $du$ the differential of $u$. The gradient of $u$ is the vector field $\nabla_g u$ defined by
\begin{equation}   \label{ec:defgrad}
    g(\nabla_g u, X)= df(X), \;\forall X\in \mathcal{X}(M).
\end{equation}
If the local components of the differential of $u$ are denoted by $u_i=\frac{\partial u}{\partial x_i}$, then the local components of the gradient of $u$ are $(\nabla_g u)^i=g^{ij}u_j$, where $g^{ij}$ are the local coefficients of $g^{-1}=(g_{ij})^{-1}$.

We recall the well known Eikonal equation: for every $x\in M$, one has
\begin{equation}\label{eikonal}
|\nabla_g d(x,\cdot)|=1\ {\rm  a.e. \ on}\ M.
\end{equation}
%In fact, relation (\ref{dist-gradient}) is valid for every point $x\in M$ outside of the cut-locus of $x_0$ (which is a null measure set).

Denote by $\nabla$ the Levi-Civita connection of $g$. We only use $\nabla_g$ (with $g$ as an index) for the gradient, so there is no danger of confusion with the Levi-Civita connection.

The divergence of a vector field $X\in\camp$ with respect to the Riemannian metric $g$ is defined as 
\begin{equation}
    \dvg(X) = \text{tr}\left(\xi\to\nabla_\xi X\right), \;\; \xi\in\camp.
\end{equation}
More precisely, for $X=X^i\frac{\prt}{\prt x_i}\in\mathcal{X}(M)$, the expression of $\dvg (X)$ in local coordinates is given by the formula
\begin{equation}  \label{ec:defdvgloc}
    \dvg \left( X^i\frac{\prt}{\prt x_i} \right) = \frac{1}{\sqrt{\det g}} \frac{\prt}{\prt x_i} \left( X^i\sqrt{\det g} \right).
\end{equation}
For later use, let us recall the next identity: for any $f\in C^\infty(M)$ and vector fields $X,Y\in\camp$ 
\begin{equation}   \label{ec:dvgformula}
    \dvg(fX)= g(\nabla_g f,X) + f\dvg(X).
\end{equation}
The Hessian operator $\hess u$ is defined as the symmetric 2-tensor 
\begin{equation}   \label{ec:defhess}
    \hess^u(X,Y) = g(\nabla_X \nabla_g u,Y), \;\forall X,Y\in TM.
\end{equation}
 The Laplace-Beltrami operator of $u$ is defined as $\Delta u = \dvg(\nabla_g u)$, which in local coordinates is expressed as
\begin{equation}   \label{ec:deflaplacianloc}
    \Delta_g u = \frac{1}{\sqrt{\det g}} \frac{\prt}{\prt x_i} \left( g^{ij}\sqrt{\det g} \frac{\prt u}{\prt x_j}\right).
\end{equation}
We define the $p-$Laplacian operator for $p\in(1,\infty)$ as $\Delta_p u = \dvg\left(\abs{\nabla_g u}^{p-2}\nabla_g u \right)$.\\ 

%The $L^p(M)$ norm of
%$\nabla_g u(x)\in T_xM$ is given by
%\begin{equation}    \label{ec:defLpnorm}
%    \|\nabla_g u\|_{L^p(M)}=\left(\displaystyle\int_M \abs{\nabla_gu}^p \dvol\right)^{{1}/{p}}.
%\end{equation}
%The space $H^1_g(M)$
%is the completion of $C_0^\infty(M)$ with respect to the norm
%\begin{equation}    \label{ec:defH1norm}
%    \|u\|_{H^1_g(M)}=\sqrt{\|u\|_{L^2(M)}^2+\|\nabla_g u\|_{L^2(M)}^2}.
%\end{equation}
Since our results will depend on the sectional curvature of $M$, we need to recall the definition of curvature of a connection $\nabla$, by introducing the $(1,3)-$tensor $\mathcal{R}:\camp\times\camp\times\camp \to \camp$, given by
$$\mathcal{R}(X,Y)Z = \nabla_X\nabla_Y Z - \nabla_Y\nabla_X Z - \nabla_{[X,Y]}Z, \;\;\text{for any}\;\; X,Y,Z\in\camp. $$
The well-known Riemann curvature tensor is defined as $R:\camp\times\camp\times\camp\times\camp \to C^\infty(M)$, 
$$R(X,Y,Z,W)=g(\mathcal{R}(Z,W)Y,X), \;\; \text{for any } X,Y,Z,W\in\camp.$$
Let $x\in M$, $\pi\subset T_xM$ a $2-$dimensional vector subspace and $\{E_1,E_2\}\subset\pi$ an orthonormal basis. The sectional curvature of $\pi$ is $$K_\pi=K_\pi(E_1,E_2)=R(E_1,E_2,E_1,E_2).$$  
We say that $M$ has non-positive (or non-negative sectional) curvature if
\begin{equation}
    K_\pi\leq 0 \;\left( \;\text{or}\;\; K_\pi \geq 0\right), \;\; \text{for any } \pi\subset T_xM.
\end{equation}
Our main tools in dealing with the extra-terms which we obtain in the Hardy inequalities will be some comparison theorems for the Laplacian and Hessian operators. For this, we first introduce a useful function. We follow the reference \cite{lee}. For any $c\in\R$, define the function $s_c:[0,\infty)\to\R$ by
\begin{equation}
s_c(r) = \left\{\begin{array}{l l}
    r, & \text{if}\; c=0 \\
    \frac{1}{\sqrt{c}}\sin{(r\sqrt{c})}, & \text{if}\; c>0 \\
    \frac{1}{\sqrt{-c}}\sinh{(r\sqrt{-c})}, & \text{if}\; c<0
\end{array}\right.
\end{equation} 
It is easy to see that
\begin{equation}
    \frac{s_c'(r)}{s_c(r)} = \left\{\begin{array}{l l}
    \frac{1}{r}, & \text{if}\; c=0 \\
    \sqrt{c}\cot{(r\sqrt{c})}, & \text{if}\; c>0 \\
    \sqrt{-c}\coth{(r\sqrt{-c})}, & \text{if}\; c<0
\end{array}\right.
\end{equation} 
Using these notations, we state now two celebrated curvature comparison results.
\begin{theorem}[Hessian comparison theorem, \cite{lee}]   \label{ec:thmHess}
    Suppose (M,g) is a $N-$dimensional Riemannian manifold, $x_0\in M$, $U$ is a neighbourhood of $x_0$ and $d$ is the distance function on $U$. If all the sectional curvatures of $M$ are bounded above by a constant $c$, then the following inequality holds in $U\setminus\{x_0\}$:
    \begin{equation}
        \hess^d \geq \frac{s_c'(d)}{s_c(d)}\pi_d,   \label{ec:hessiancomparison}
    \end{equation}
    where for any $q\in U\setminus\{x_0\}$, $\pi_d:T_qM \to T_qM$ is the orthogonal projection onto the tangent space of the level set of $d$ (or, equivalently, onto the orthogonal complement of $\prt_d|_q$).
\end{theorem}
\begin{theorem}[Laplacian comparison theorem, \cite{lee}]   \label{ec:thmLapl}
    Suppose (M,g) is a $N-$dimensional Riemannian manifold, $x_0\in M$, $U$ is a neighbourhood of $x_0$ and $d$ is the distance function on $U$. If all the sectional curvatures of $M$ are bounded above by a constant $c$, then the following inequality holds in $U\setminus\{x_0\}$:
    \begin{equation}
        \Delta d \geq (N-1)\frac{s_c'(d)}{s_c(d)},     \label{ec:laplaciancomparison}
    \end{equation}
    where $s_c$ is defined above.
\end{theorem}
The inequalities above are understood in the distributional sense. 

\begin{remark}
\begin{enumerate}
    \item[1)] Notice that, since $\pi_d$ is an orthogonal projection, we have: 
\begin{equation}   \label{ec:proj1}
    g(\pi_d X,X) = g(\pi_d^2 X,X)= g(\pi_d X,\pi_d X) = \abs{\pi_d X}^2 \geq 0, \;\; \text{for any } X\in\camp.
\end{equation}
   \item [2)] Moreover, we have
   \begin{equation}   \label{ec:proj2}
       \pi_d X = X-\frac{g(X,\nabla_g d)}{g(\nabla_g d,\nabla_g d)}\nabla_g d, \;\; \text{for any } X\in\camp.
   \end{equation}
\end{enumerate}
\end{remark}

In order to find sharp estimates in our inequalities, we will address a special class of manifolds, namely Cartan-Hadamard manifolds.
\begin{definition}
     If $M$ is a complete, simply connected Riemannian manifold with non-positive sectional curvature, then it is called a Cartan–Hadamard manifold. 
\end{definition}
\begin{remark}   \label{ec:rmk1}
If $M$ is a Cartan-Hadamard manifold and $d$ is the distance function to $x_0\in M$, since the sectional curvatures are all bounded from above by $c\leq 0$, then, by Theorems \ref{ec:thmHess} and \ref{ec:thmLapl}, we have the following comparison principles:
\begin{align}
    1)& \hess^d(X,X) \geq 0,\;\; \forall X\in\camp; \\
    2)& d \Delta_g d - (N-1) \geq 0.
\end{align}
\end{remark}

\section{Main results}   \label{main-results}
Let $M$ be a complete, simply connected Riemannian manifold of dimension $N\geq 3$ with Riemannian metric $g$, $n\geq 2$ integer, $a_1,...a_n$ fixed points in $M$. Denote by 
$$d_i= d_i(x) := d(x,a_i)$$
the distance function from the point $x\in M$ to the fixed point $a_i$, for any $i=1,...,n$. The following notations will be consistently used throughout the paper:
\begin{equation}
        v:=\sum_{i=1}^n \frac{\nabla_g d_i}{d_i} \;\;\;\text{and}\;\;\; G_{ij}:=g(\nabla_g d_i, \nabla_g d_j).
    \end{equation}
We introduce the following potential:
\begin{align}   \label{ec:defV}
    &V_{p,a_1,...,a_n} := C_1(n,p) \sum_{1\leq i<j\leq n} \abs{\frac{\nabla_g d_i}{d_i} - \frac{\nabla_g d_j}{d_j}}^2 \abs{v}^{p-2} + C_2(n,p) \Bigg[ \sum_{i=1}^n  \frac{d_i\Delta_g d_i - (N-p+1)}{d_i^2} \abs{v}^{p-2}  \notag\\
        & - (p-2) \sum_{i,j, k =1}^n \frac{G_{ij}G_{ik}}{d_i^2d_jd_k} \abs{v}^{p-4} + (p-2) \sum_{\substack{1\leq k \leq n,\\ 1\leq i<j\leq n}} \frac{\hess^{d_i} (\nabla_g d_k,\nabla_g d_j) + \hess^ {d_j} ( \nabla_g d_k,\nabla_g d_i )}{d_id_jd_k}   \abs{v}^{p-4} \Bigg]
\end{align}
where
\begin{align}   \label{ec:defC}
    C_1(n,p) = \frac{(N-p)^p}{n^p(p-1)^{p-1}} \;\;\; \text{and}\;\;\; C_2(n,p)= \frac{(N-p)^{p-1}}{n^{p-1}(p-1)^{p-1}}.  
\end{align}
We denote by
\begin{equation}
    \widetilde{V}_{p,a_1,...,a_n} := \sum_{1\leq i<j\leq n} \abs{\frac{\nabla_g d_i}{d_i} - \frac{\nabla_g d_j}{d_j}}^2 \abs{v}^{p-2}
\end{equation}
the leading component of the potential $V_{p,a_1,...,a_n}$. Asymptotically,  $$\widetilde{V}_{p,a_1,...,a_n} \sim d_i^{-p} \;\;\;\; \text{when $x$ is close to $a_i$, for all $i=1,...,n$.}$$
It is clear that in the case $p=2$, the potential $\widetilde{V}_{2,a_1,...,a_n}$ coincides with the singular potential from the right hand-side of\eqref{ec:cazacu-zuazua} if we restrict it to $\R^N$, while $V_{2,a_1,...,a_n}$ coincides with the potential from the right hand-side of \eqref{ec:kristaly-p=2}. Nevertheless, for any $1<p<N$, $V_{p,a_1,a_2}$ recovers the bipolar case $n=2$ from \eqref{ec:cazacu-rugina} when restricted to $\R^N$.
\vspace{0.1cm}

Our first result is the following multipolar Hardy-type inequality.
\begin{theorem}   \label{ec:thm1}
    Let $(M^N,g)$ be a complete Riemannian manifold of dimension $N\geq 3$ and $a_1,...,a_n$ points in $M$, where $n\geq 2$. Then the following inequality holds for any $1<p<N$ and any $u\in C^\infty_c(M)$:
    \begin{align}
        \int_M \abs{\nabla_g u}^p \dvol & \geq \int_M \V \abs{u}^p \dvol.
    \end{align} 
\end{theorem}
It is not certain that $V$ is positive for any number of points $a_i$ and any $1<p<N$. This constrains us to consider further the case of only two points $a_1$ and $a_2$.  The next proposition gives us a better expression of the potential from \eqref{ec:defV} in the bipolar case.
\begin{prop}    \label{ec:prop1}
    Let $a_1,a_2$ be points in $M$. Then
    \begin{align}
        V_{p,a_1,a_2} & = C_1(2,p) \abs{\frac{\nabla_g d_1}{d_1} - \frac{\nabla_g d_2}{d_2}}^2 \abs{v}^{p-2} + C_2(2,p) \Bigg[ \sum_{i=1}^2 \frac{d_i\Delta_g d_i - (N-1)}{d_i^2} \abs{v}^{p-2} \notag\\
    & + 2(p-2) \frac{1-G_{12}^2}{d_1^2d_2^2} \abs{v}^{p-4} + (p-2) \left( \frac{\hess^{d_1} (\nabla_g d_2,\nabla_g d_2)}{d_1d_2^2} + \frac{\hess^ {d_2} (\nabla_g d_1,\nabla_g d_1)}{d_1^2d_2} \right) \abs{v}^{p-4} \Bigg]
    \end{align}
\end{prop}
We define the functional space $\mathcal{D}^{1,p}(M)$ as
\begin{equation}
    \mathcal{D}^{1,p}(M):=\Big\{u\in\mathcal{D}'(M) \;\Big|\; \int_M \abs{\nabla_g u}^p \dvol < \infty \Big\}.
\end{equation}
When restricting to the bipolar case in Theorem \ref{ec:thm1}, using Proposition \ref{ec:prop1} we obtain the following corollary.
\begin{corollary}   \label{ec:cor1}
Let $(M^N,g)$ be a complete Riemannian manifold of dimension $N\geq 3$ and $a_1$, $a_2$ points in $M$. Then the following inequality holds for any $1<p<N$ and any $u\in C^\infty_c(M)$:
    \begin{align}   \label{ec:ineq-bipolar}
        \int_M \abs{\nabla_g u}^p \dvol & \geq\int_M V_{p,a_1,a_2} \abs{u}^p \dvol.
    \end{align}
    Moreover, if $M$ is a Cartan-Hadamard manifold, for any $2\leq p<N$ the right-hand side of \eqref{ec:ineq-bipolar} is positive and the constant $1$ is sharp in \eqref{ec:ineq-bipolar}. This constant is achieved in the energy space $\mathcal{D}^{1,p}(M)$ by the minimizers
    \begin{equation}   \label{ec:defphi-thm}
        \phi(x)= C \; d_1^\frac{p-N}{2(p-1)}d_2^\frac{p-N}{2(p-1)}, \;\; C\in\R,
    \end{equation}
    unless the case $p=2$ when the constant is not achieved.
\end{corollary}
Notice that Corollary \ref{ec:cor1} gives a dichotomy between the case $2<p<N$ and $p=2$, since in \cite{kristaly1} the sharp constant in the inequality is not achieved.  \\

Consequently, we obtain another sharp result for Cartan-Hadamard manifolds.
\begin{corollary}   \label{ec:cor2}
If $M$ is a $N$-dimensional Cartan-Hadamard manifold with sectional curvatures bounded above by a negative constant $c=-R^2$, $R>0$, then for any $2\leq p<N$ and $u\in C^\infty_c(M)$ it holds
    \begin{align}     \label{ec:ineq-CH}
        \int_M \abs{\nabla_g u}^p \dvol & \geq C_1(2,p) \int_M \widetilde{V}_{p,a_1,a_2} \abs{u}^p \dvol + (N-1)C_2(2,p) \sum_{i=1}^2 \int_M \frac{3R^2}{\pi^2+R^2d_i^2} \abs{v}^{p-2}\abs{u}^p \dvol   \notag\\
        & + (p-2)C_2(2,p)\sum_{i=1}^2 \int_M \frac{1-G_{12}^2}{d_1^2d_2^2} \left[ 
2+\frac{3d_i^2R^2}{\pi^2+R^2d_i^2} \right] \abs{v}^{p-4} \abs{u}^p \dvol.
    \end{align}
In particular, 
\begin{equation}   \label{ec:ineq-CH2}
     \int_M \abs{\nabla_g u}^p \dvol  \geq C_1(2,p) \int_M \widetilde{V}_{p,a_1,a_2} \abs{u}^p \dvol,
\end{equation}
where the constant $C_1(2,p)=(p-1)\left(\frac{N-p}{2(p-1)}\right)^p$ is sharp and not attained in the energy space $\mathcal{D}^{1,p}(M)$.
\end{corollary}
We note that even though the inequality \eqref{ec:ineq-CH} is weaker than \eqref{ec:ineq-bipolar} due to the lower-bound estimates on the remainder terms, it reveals the importance of the curvature bound $c=-R^2$, which enhances the inequality as the curvature becomes more negative.

\section{Proof of Main results}    \label{proof}
The proof of Theorem \ref{ec:thm1} relies on an adaptation of the method of supersolutions introduced by Allegretto and Huang in \cite{picone} to Riemannian manifolds. The proof follows the same steps as in Theorem 2.1 in \cite{picone}. Here is the adapted version in the setting of Riemannian manifolds:
\begin{prop} \label{prop2}
	Let $N\geq 3$, $1<p<\infty$. If there exists a  function $\phi>0$ in $M$ such that $\phi \in C^\infty_c(M\setminus\bigcup_{i=1}^n \{a_i\}\cup \cut\{a_i\})$ and  
	\begin{equation}\label{ine1}
		-\Delta_{p,g}\phi \geq \mu V \phi^{p-1}, \quad \forall  x\in M\setminus\{a_1,...,a_n\},           
	\end{equation}
	where $V>0$, with $V\in L_{loc}^1(M)$, is a given multi-singular potential with the poles  $a_1, \ldots, a_n$ and $\mu>0$, then
	\begin{equation}
		\int_{M} \abs{\nabla_g u}^p \dvol \geq \mu\int_{M} V\abs{u}^p \dvol, \;\; \forall u\in C_c^\infty(M).     \notag
	\end{equation}
\end{prop}   

\begin{proof} [\textbf{Proof of Theorem \ref{ec:thm1}.}]
We want to find a function $\phi>0$, a constant $\mu$ depending on $N$, $n$ and $p$ and a potential $V\in L^1_{loc}(M)$, with singularities in the points $a_1$, $a_2$,..., $a_n$, which satisfy the identity 
\begin{equation}
    -\frac{\Delta_{p,g}\phi}{\phi^{p-1}} = \mu V, \;\; a.e. \text{for}\; x\in M\setminus\{a_1,...,a_n\}.        \notag
\end{equation}
We consider the functions
\begin{equation} 
\phi_i=d_i^\be, \quad\quad i=1,...,n   \notag
\end{equation}
where $\be$ is negative, aimed to depend on $N$, $n$ and $p$, that  will be precised later. We introduce 
\begin{equation}    \label{ec:defphi}
   \phi= \prod_{i=1}^n \phi_i = \prod_{i=1}^n d_i^\be.  
\end{equation}
We compute the $p$-Laplacian of $\phi$ in \eqref{ec:defphi} in several steps. First, we note that
\begin{equation}   \label{ec:gradphi}
    \nabla_g\phi= \phi\sum_{i=1}^n \frac{\nabla_g \phi_i}{\phi_i} = \be \phi \sum_{i=1}^n \frac{\nabla_g d_i}{d_i}.   
\end{equation}
To simplify the notation, denote by
\begin{equation}   
     v:= \sum_{i=1}^n \frac{\nabla_g d_i}{d_i}. \label{ec:defvector}
\end{equation}
Using the definition of the $p$-Laplacian operator with respect to the metric $g$ and identity \eqref{ec:dvgformula}, we obtain
\begin{align}
    \Delta_{p,g} \phi & = \dvg_g\Big( \abs{\nabla_g\phi}^{p-2}\nabla_g\phi \Big) \notag \\
    & = \be\abs{\be}^{p-2} \dvg_g\Big(\phi^{p-1} \abs{v}^{p-2}v  \Big) \notag \\
    & = \be\abs{\be}^{p-2} \Bigg[ g\left(\nabla_g\Big(\phi^{p-1}\Big)\abs{v}^{p-2} , v\right) + \phi^{p-1}\abs{v}^{p-2} \dvg_g(v) + \phi^{p-1}g\left(\nabla_g\Big(\abs{v}^{p-2}\Big),v\right) \Bigg]  \notag\\
    & = \be\abs{\be}^{p-2}\phi^{p-1} \Bigg[ (p-1)\be \abs{v}^p + \abs{v}^{p-2} \dvg_g(v) + g\left(\nabla_g\Big(\abs{v}^{p-2}\Big),v\right) \Bigg].  \notag
\end{align}
Hence, we denote
\begin{equation}  
     -\frac{\Delta_{p,g}\phi}{\phi^{p-1}} =: V,    \notag
\end{equation}
where 
\begin{equation}
    V = -\be\abs{\be}^{p-2} \Bigg[ (p-1)\be\abs{v}^p + \abs{v}^{p-2}\dvg_g(v) + g\left(\nabla_g\left(\abs{v}^{p-2}\right),v\right) \Bigg]. \;    \label{ec:potential}
\end{equation}
Next, we compute explicitly the three terms in (\ref{ec:potential}). Taking the modulus of $v$, we get
\begin{align}
     \abs{v}^2 & = \sum_{i=1}^n \frac{1}{d_i^2} + 2\sum_{1\leq i<j\leq n} \frac{g(\nabla_g d_i, \nabla_gd_j)}{d_id_j},  \notag\\
     & = -\sum_{1\leq i<j\leq n} \abs{\frac{\nabla_g d_i}{d_i} - \frac{\nabla_g d_j}{d_j}}^2 + n\sum_{i=1}^n \frac{1}{d_i^2}.   \label{ec:modulv^2}
\end{align}
where we are taking into account that, for every $i=1,...,n$, we have $g(\nabla_g d_i,\nabla_g d_i) = \abs{\nabla_g d_i}^2 = 1$, since $d_i$ satisfies the Eikonal equation.
The second term in \eqref{ec:potential} yields to
\begin{equation}
    \dvg_g(v) = \sum_{i=1}^n \frac{d_i\Delta_g d_i -1}{d_i^2}.   \label{ec:termen2}
\end{equation}
The third term in \eqref{ec:potential} will become
\begin{align}
    g & \left( \nabla_g\left(\abs{v}^{p-2}\right),v\right) = \frac{p-2}{2} \abs{v}^{p-4} g\left(\nabla_g\abs{v}^2,v\right)   \notag\\ 
    & =  \frac{p-2}{2} \abs{v}^{p-4} \sum_{i,j,k=1}^n  g\left( \nabla_g \Bigg( \frac{g(\nabla_g d_i, \nabla_gd_j)}{d_id_j} \Bigg) , \frac{\nabla_g d_k}{d_k} \right)   \notag\\
    & = \frac{p-2}{2} \abs{v}^{p-4} \sum_{i,j,k=1}^n  g\left( -\frac{\nabla_g(d_id_j)G_{ij}}{d_i^2d_j^2} + \frac{\nabla_g(G_{ij})}{d_id_j} , \frac{\nabla_g d_k}{d_k} \right)   \notag\\
    & = \frac{p-2}{2} \abs{v}^{p-4} \sum_{i,j,k=1}^n \Bigg[ -\frac{G_{ij}}{d_i^2d_j^2d_k} g\left( d_j\nabla_g d_i + d_i \nabla_g d_j , \nabla_g d_k \right) + \frac{1}{d_id_jd_k} g\left(\nabla_g(G_{ij}), \nabla_g d_k \right) \Bigg]  \notag\\
    & = \frac{p-2}{2} \abs{v}^{p-4} \sum_{i,j,k=1}^n \Bigg[ -\frac{G_{ij}G_{ik}}{d_i^2d_jd_k}  - \frac{G_{ij}G_{jk}}{d_id_j^2d_k}   + \frac{1}{d_id_jd_k} g\left(\nabla_g(G_{ij}), \nabla_g d_k \right) \Bigg]  \notag\\
    & = -(p-2) \abs{v}^{p-4} \sum_{i,j,k=1}^n \frac{G_{ij}G_{ik}}{d_i^2d_jd_k} + (p-2) \abs{v}^{p-4} \sum_{\substack{1\leq k \leq n,\\ 1\leq i<j\leq n}} \frac{g\left(\nabla_g(G_{ij}), \nabla_g d_k \right)}{d_id_jd_k}.  \label{ec:termen3}
\end{align}
Note that the last term can be written in the following way, using the definition of the gradient and the Hessian operator
\begin{align}
    g\left(\nabla_g(G_{ij}), \nabla_g d_k \right) & = g\left( \nabla_g \Big( g(\nabla_g d_i,\nabla_g d_j) \Big) , \nabla_g d_k \right)   \notag\\
    & = \nabla_g d_k \left(g(\nabla_g d_i,\nabla_g d_j)\right)  \notag\\
    & = g\left( \nabla_{\nabla_g d_k}\nabla_g d_i, \nabla_g d_j \right) + g\left( \nabla_{\nabla_g d_k}\nabla_g d_j, \nabla_g d_i \right)  \notag\\
    & = \hess^{d_i} (\nabla_g d_k,\nabla_g d_j) + \hess^ {d_j} ( \nabla_g d_k,\nabla_g d_i ).   \label{ec:termen3'}
\end{align}
Returning to $V$ in \eqref{ec:potential} and using \eqref{ec:termen2}, \eqref{ec:termen3} and \eqref{ec:termen3'}, the potential reduces to
\begin{align}  \label{ec:potential2}
    V  &= -\be\abs{\be}^{p-2} \abs{v}^{p-2} \Bigg[ (p-1)\be \left( -\sum_{1\leq i<j\leq n} \abs{\frac{\nabla_g d_i}{d_i} - \frac{\nabla_g d_j}{d_j}}^2 + n\sum_{i=1}^n \frac{1}{d_i^2} \right) + \sum_{i=1}^n \frac{d_i\Delta_g d_i - 1}{d_i^2}  \Bigg] \notag\\
    & +(p-2)\be\abs{\be}^{p-2} \abs{v}^{p-4} \Bigg[ -\sum_{i,j,k=1}^n \frac{G_{ij}G_{ik}}{d_i^2d_jd_k} + \sum_{\substack{1\leq k \leq n,\\ 1\leq i<j\leq n}} \frac{\hess^{d_i} (\nabla_g d_k,\nabla_g d_j) + \hess^ {d_j} ( \nabla_g d_k,\nabla_g d_i )}{d_id_jd_k} \Bigg].  
\end{align}
Choosing $\be=\frac{p-N}{n(p-1)}$, we obtain
\begin{align}   
    V  &= \left( \frac{N-p}{n(p-1)} \right)^{p-1} \abs{v}^{p-2} \Bigg[ \left(\frac{N-p}{n}\right) \sum_{1\leq i<j\leq n} \abs{\frac{\nabla_g d_i}{d_i} - \frac{\nabla_g d_j}{d_j}}^2 + \sum_{i=1}^n \frac{p-N}{d_i^2} + \sum_{i=1}^n \frac{d_i\Delta_g d_i - 1}{d_i^2}  \Bigg] \notag\\
    & +(p-2)\left( \frac{N-p}{n(p-1)} \right)^{p-1} \abs{v}^{p-4} \Bigg[ -\sum_{i,j,k=1}^n \frac{G_{ij}G_{ik}}{d_i^2d_jd_k} + \sum_{\substack{1\leq k \leq n,\\ 1\leq i<j\leq n}} \frac{\hess^{d_i} (\nabla_g d_k,\nabla_g d_j) + \hess^ {d_j} ( \nabla_g d_k,\nabla_g d_i )}{d_id_jd_k} \Bigg]. \notag 
\end{align}
The final form of $V$ is the following
\begin{align}   \label{ec:potential3}
    V  &= (p-1)\left( \frac{N-p}{n(p-1)} \right)^p \sum_{1\leq i<j\leq n} \abs{\frac{\nabla_g d_i}{d_i} - \frac{\nabla_g d_j}{d_j}}^2 \abs{v}^{p-2} + \left( \frac{N-p}{n(p-1)} \right)^{p-1} \sum_{i=1}^n \frac{d_i\Delta_g d_i - (N-p+1)}{d_i^2} \abs{v}^{p-2} \notag\\
    & +(p-2)\left( \frac{N-p}{n(p-1)} \right)^{p-1} \Bigg[ -\sum_{i,j,k=1}^n \frac{G_{ij}G_{ik}}{d_i^2d_jd_k} +\sum_{\substack{1\leq k \leq n,\\ 1\leq i<j\leq n}} \frac{\hess^{d_i} (\nabla_g d_k,\nabla_g d_j) + \hess^ {d_j} ( \nabla_g d_k,\nabla_g d_i )}{d_id_jd_k} \Bigg] \abs{v}^{p-4}.     
\end{align}
It is clear, according to \eqref{ec:defV} and \eqref{ec:defC}, that
\begin{equation}
    V = V_{p,a_1,...,a_n}.
\end{equation}
By Proposition \ref{prop2}, with $V = V_{p,a_1,...,a_n}$ and constant $\mu=1$, Theorem \ref{ec:thm1} is proved.
\end{proof}

We are now in the position of proving the simplified form of $V_{p,a_1,...,a_n}$ when we restrict only to two poles $a_1$ and $a_2$. 
\begin{proof} [\textbf{Proof of Proposition \ref{ec:prop1}.}]
    We will use the notations of $C_1$ and $C_2$ from \eqref{ec:defC} for $n=2$. The expression of $V$ in \eqref{ec:potential3} becomes
\begin{align}
    V & = C_1(2,p) \abs{\frac{\nabla_g d_1}{d_1} - \frac{\nabla_g d_2}{d_2}}^2 \abs{v}^{p-2} + C_2(2,p) \sum_{i=1}^2 \frac{d_i\Delta_g d_i - (N-p+1)}{d_i^2} \abs{v}^{p-2} \notag\\
    & +(p-2)C_2(2,p) \sum_{k=1}^2 \Bigg[ -\sum_{i,j=1}^2 \frac{G_{ij}G_{ik}}{d_i^2d_jd_k} + \frac{\hess^{d_1} (\nabla_g d_k,\nabla_g d_2) + \hess^ {d_2} (\nabla_g d_k,\nabla_g d_1)}{d_1d_2d_k} \Bigg] \abs{v}^{p-4}. \notag
\end{align}     
We look at the mixed term containing $G_{ij}$ and $G_{ik}$. Recall that $G_{ij}=g(\nabla_g d_i,\nabla_g d_j)$ and $\abs{\nabla_g d_i}=1$ for any $i,j=\overline{1,2}$. Then,
\begin{align}
    \sum_{i,j,k=1}^2 \frac{G_{ij}G_{ik}}{d_i^2d_jd_k} & = \frac{1}{d_1^4} + \frac{1}{d_2^4} + 2\frac{G_{ij}}{d_1^3d_2} + 2\frac{G_{ij}}{d_1d_2^3} + 2\frac{G_{ij}^2}{d_1^2d_2^2}   \notag\\
    & = \frac{1}{d_1^4d_2^4} \left[ d_1^4 + d_2^4 + 2d_1^3d_2 G_{12} + 2d_1d_2^3 G_{12} + 2d_1^2d_2^2 G_{12}^2 \right]  \notag\\
    & = \frac{1}{d_1^4d_2^4} \left[ \left(d_1^2 + d_2^2\right)^2 + 2d_1d_2 G_{12}\left( d_1^2 + d_2^2 \right) -  2d_1^2d_2^2\left( 1 - G_{12}^2\right) \right]  \notag\\
    & = \abs{v}^2 \left( \frac{1}{d_1^2}+\frac{1}{d_2^2} \right) - 2\frac{1-G_{12}^2}{d_1^2d_2^2}.
\end{align}
We also notice that the mixed terms with Hessians will be 0, using the symmetry of the Hessian operator:
\begin{align}
    \hess^{d_1}(\nabla_g d_1, \nabla_g d_2) & = \hess^{d_1}(\nabla_g d_2, \nabla_g d_1) \notag\\
    & = g\left(\nabla_{\nabla_g d_2} \nabla_g d_1,\nabla_g d_1\right)  \notag\\
    & = \nabla_g d_2\left(g(\nabla_g d_1, \nabla_g d_1) \right)  \notag\\
    & = \nabla_g d_2\left(\abs{\nabla_g d_1}^2 \right) = 0.
\end{align}
Similarly, $\hess^{d_2}(\nabla_g d_2, \nabla_g d_1)=0$. Hence, the potential $V$ becomes:
\begin{align}
    V & = C_1(2,p) \abs{\frac{\nabla_g d_1}{d_1} - \frac{\nabla_g d_2}{d_2}}^2 \abs{v}^{p-2} + C_2(2,p) \sum_{i=1}^2 \frac{d_i\Delta_g d_i - (N-p+1)}{d_i^2} \abs{v}^{p-2}   \notag\\
    & +(p-2) C_2(2,p) \Bigg[ - \left( \frac{1}{d_1^2}+\frac{1}{d_2^2} \right) \abs{v}^2 + 2\frac{1-G_{12}^2}{d_1^2d_2^2} + \frac{\hess^{d_1} (\nabla_g d_2,\nabla_g d_2)}{d_1d_2^2} + \frac{\hess^ {d_2} (\nabla_g d_1,\nabla_g d_1)}{d_1^2d_2} \Bigg] \abs{v}^{p-4}  \notag\\
    & = C_1(2,p) \abs{\frac{\nabla_g d_1}{d_1} - \frac{\nabla_g d_2}{d_2}}^2  \abs{v}^{p-2} + C_2(2,p) \sum_{i=1}^2 \frac{d_i\Delta_g d_i - (N-1)}{d_i^2}  \abs{v}^{p-2}    \notag\\
    & + 2(p-2)C_2(2,p) \frac{1-G_{12}^2}{d_1^2d_2^2} \abs{v}^{p-4} + (p-2)C_2(2,p) \left( \frac{\hess^{d_1} (\nabla_g d_2,\nabla_g d_2)}{d_1d_2^2} + \frac{\hess^ {d_2} (\nabla_g d_1,\nabla_g d_1)}{d_1^2d_2} \right) \abs{v}^{p-4}.   \notag
\end{align}
\end{proof}

\begin{proof} [\textbf{Proof of Corollary \ref{ec:cor1}.}]
Inequality \eqref{ec:ineq-bipolar} is a simple consequence of Theorem \ref{ec:thm1} and Proposition \ref{ec:prop1}.   \vspace{0.2cm}\\ 
On Cartan-Hadamard manifolds, the positiveness of the right-hand side of \eqref{ec:ineq-bipolar} for $2\leq p<N$ is deduced by Laplacian comparison principle - Theorem \ref{ec:thmLapl} - and Hessian comparison principle - Theorem \ref{ec:thmHess} - (see also Remark \ref{ec:rmk1}), coupled with the Cauchy-Schwarz inequality. \\
In order to prove that the constant $1$ is sharp in \eqref{ec:ineq-bipolar} and actually attained in $\mathcal{D}^{1,p}(M)$ for $2<p<N$, we prove that
\begin{equation}
    \int_M \abs{\nabla_g \phi}^p \dvol = \int_M V_{p,a_1,a_2} \phi^p \dvol, 
\end{equation}
for $\phi$ defined in \eqref{ec:defphi-thm}. This is done by integration by parts, but first, we need to prove that $\phi\in \mathcal{D}^{1,p}(M)$, i.e. that $\|\nabla_g\phi\|_p$ is finite. Using the definition of $\phi$, for $\be=\frac{p-N}{2(p-1)}$ and using \eqref{ec:gradphi}, we get
\begin{equation}
    \int_M \abs{\nabla_g \phi}^p \dvol = \abs{\be}^p \int_M \phi^p \abs{\sum_{i=1}^n\frac{\nabla_g d_i}{d_i}}^p \dvol.   \notag 
\end{equation}
We have to divide the integral in different regions in $M$. More precisely, let $r>0$ sufficiently small, say $r<\frac{\dist(a_1,a_2)}{4}$, and $R>\dist(a_1,a_2)$. Fix $a$ to be a point with the property that $\dist(a,a_1)=\dist(a,a_2)$ situated on a geodesic which connects $a_1$ and $a_2$. 
We consider the balls $B_{a_i}(r)$ of radius $r$ around the points $a_i$ and the set of points $A:=\{x\in M \;|\; \dist(a,a_i)>R\}$, for any $i=1,2$. Note that from the choice of $r$ and $R$, the set $A$ contains both balls $B_r(a_i)$. Now, the integral above is 
\begin{align}
    \int_M \abs{\nabla_g \phi}^p \dvol & = \int_{M\setminus A} \abs{\nabla_g \phi}^p \dvol + \int_{A\setminus \cup_{i=1}^2 B_r(a_i)} \abs{\nabla_g \phi}^p \dvol + \int_{\cup_{i=1}^2 B_r(a_i)} \abs{\nabla_g \phi}^p \dvol    \notag\\
    & := I_1 + I_2 + I_3.          \label{integrals}
\end{align}
Since the set $A\setminus \cup_{i=1}^2 B_r(a_i)$ is compact, the integral $I_2$ is finite. From now on, by $\simeq$ between two quantities, we understand that we disregard the constants appearing in computations and that the asymptotic behavior of the quantities involved is the same. Now for any $i=1,2$, by the co-area formula, we have
\begin{equation}
    \int_{B_r(a_i)} \abs{\nabla_g \phi}^p \dvol \simeq \int_{B_r(a_i)} d_i^{p(\be-1)}\dvol \simeq \int_0^r t^{p(\be-1)+N-1} dt.   \notag
\end{equation}
This integral is finite if 
\begin{equation}
    p(\be-1)+N >0  \;\;\; \iff \;\;\; \frac{(p-2)(N-p)}{2(p-1)}>0.   \notag
\end{equation}
Hence, the integral $I_3$ is finite for any $2\leq p<N$. For the first integral, we need to make some estimates first. Let us notice, from the properties of the distance function and the point $a$, that for any $x\in M\setminus A$
\begin{equation}
    d(x,a_i) \leq d(x,a) + d(a,a_i) \leq 2d(x,a)    \notag
\end{equation}
and 
\begin{equation}
    d(x,a_i) \geq d(x,a) - d(a,a_i) \geq d(x,a) - \frac{R}{2} = \frac{d(x,a)}{2} + \frac{d(x,a)-R}{2} \geq \frac{d(x,a)}{2}.   \notag
\end{equation}
Hence, we conclude that 
\begin{equation}
    \frac{d(x,a)}{2} \leq d(x,a_i) \leq 2 d(x,a).  \notag
\end{equation}
So we can estimate the integral $I_1$ in terms of the distance from $x$ to $a$, denoted by $d:=d(x,a)$.
\begin{equation}
    \int_{M\setminus A} \abs{\nabla_g \phi}^p \dvol  \simeq \int_{M\setminus A} d^{p(2\be-1)} \dvol  \simeq \int_R^\infty t^{p(2\be-1)+N-1} \dvol.    \notag
\end{equation}
This integral is finite if 
\begin{equation}
    p(2\be-1)+N <0 \;\;\;\iff\;\;\; \frac{p-N}{p-1} <0.   \notag
\end{equation}
Thus, the integral $I_1$ is finite for $p>1$.
Combining the above, we obtain that for $2\leq p<N$ the integrals $I_1$, $I_2$ and $I_3$ are finite, hence $\|\nabla_g \phi\|_p$ is finite. Now, using integration by parts, we get
\begin{equation}
    \int_M V_{p,a_1,a_2} \abs{\phi}^p \dvol = -\int_M \frac{\Delta_p \phi}{\phi^{p-1}} \abs{\phi}^p \dvol = \int_M \abs{\nabla_g \phi}^p \dvol,  \notag
\end{equation}
which concludes the proof.
\end{proof}

Now we will investigate the result on Cartan-Hadamard manifolds, where we are able to prove a sharp inequality.
\begin{proof} [\textbf{Proof of Corollary \ref{ec:cor2}}]
The inequality is derived from Corollary \ref{ec:cor1} and the following estimates. Recall that the sectional curvatures are bounded from above by a constant $c=-R^2$, where $R>0$.\\
First, by the inequality (see \cite[Theorem 1.4]{kristaly3})
\begin{equation}    \label{ec:ineg-curbura}
    t\coth(t) - 1 \geq \frac{3t^2}{\pi^2+t^2}, \;\;\;\forall t>0
\end{equation}
and using \eqref{ec:laplaciancomparison}, we have that
\begin{align}    \label{ec:ineqcor-1}
    d_i\Delta_g d_i - (N-1) & \geq (N-1) \Big[ d_iR \coth(d_i R) -1 \Big]   \notag\\
    & \geq (N-1) \frac{3R^2 d_i^2}{\pi^2 + R^2d_i^2}.    
\end{align}
Next, we estimate the Hessian using \eqref{ec:hessiancomparison} and \eqref{ec:proj2}:
\begin{align}    \label{ec:ineqcor-2}
    \hess^{d_1}(\nabla_g d_2, \nabla_g d_2) & \geq d_1R\coth(d_1R) \;g(\pi_{d_1}\nabla_g d_2, \nabla_g d_2)   \notag\\
    & \geq \Bigg[ 1 + \frac{3R^2 d_i^2}{\pi^2 + R^2d_i^2}\Bigg] g\left(\nabla_g d_2 -\frac{g(\nabla_g d_2, \nabla_g d_1)}{g(\nabla_g d_1,\nabla_g d_1)}\nabla_g d_1 , \nabla_g d_2 \right)    \notag\\
    & = \Bigg[ 1 + \frac{3R^2 d_i^2}{\pi^2 + R^2d_i^2}\Bigg] \left( 1-G_{12}^2 \right). 
\end{align}
Combining \eqref{ec:ineqcor-1} and \eqref{ec:ineqcor-2}, we get the desired \eqref{ec:ineq-CH}.

It remains to prove the sharpness of the constant $C_1(2,p)$ in \eqref{ec:ineq-CH2}. Let $\eps>0$ small enough such that $B_{2\sqrt{\eps}}(a_1)\cap B_{2\sqrt{\eps}}(a_2)=\emptyset$. For $i=1,2$, let
\begin{equation}
    u_\eps(x)=\left\{ \begin{array}{lll}
	\frac{\log\left(\frac{d_i(x)}{\eps^2}\right)}{\log\left(\frac{1}{\eps}\right)}d_i(x)^\frac{p-n}{2(p-1)}, & {\rm if} & x\in A_i[\eps^2,\eps]; \\
	\frac{2\log\left(\frac{\sqrt{\eps}}{d_i(x)}\right)}{\log\left(\frac{1}{\eps}\right)}d_i(x)^\frac{p-n}{2(p-1)}, & {\rm if} & x\in A_i[\eps,\sqrt{\eps}];  \\
	0, & &\hbox{ otherwise,}
	\end{array} \right. .
\end{equation}
Since $u_\eps$ has compact support for any $\eps$, we can use it as a test function in the inequality and, since all the other terms are positive, we get that 
\begin{equation}
    C_1(2,p) \leq \frac{\int_M \abs{\nabla u_\eps}^p \dvol}{\int_M \abs{v}^{p-2}\abs{\frac{\nabla_g d_1}{d_1}-\frac{\nabla_g d_2}{d_2}}^2 \abs{u_\eps}^p \dvol} = \frac{I_\eps}{J_\eps - 2 K_\eps},   \notag
\end{equation}
where we denoted:
\begin{equation}
    I_\eps = \int_M \abs{\nabla_g u_\eps}^p \dvol, \;\;\;
    J_\eps = \int_M \abs{v}^{p-2} \left(\frac{1}{d_1^2}+\frac{1}{d_2^2}\right) \abs{u_\eps}^p \dvol, \;\;\;
    K_\eps = \int_M \abs{v}^{p-2} \frac{g(\nabla_g d_1,\nabla_g d_2)}{d_1d_2} \abs{u_\eps}^p \dvol.  \notag
\end{equation}
Similar to the work in \cite{kristaly1}, we can show by direct computations that
\begin{equation}
    I_\eps - C_1(2,p) J_\eps = \mathcal{O}(1),\;\;\; K_\eps = \mathcal{O}(\sqrt{\eps}) \;\; \text{and}\;\;\lim_{\eps\to 0} J_\eps = +\infty.  \notag
\end{equation}
Finally, we conclude that
\begin{equation}
    C_1(2,p) \leq \frac{J_\eps + \mathcal{O}(1)}{J_\eps + \mathcal{O}(\sqrt{\eps})} \stackrel{\eps\to 0}{\longrightarrow} C_1(2,p),
\end{equation}
and the proof is finished.
\end{proof}

\section{Bipolar Hardy inequalities in constant curvature setting}    \label{constant-curvature}
In this section we want to explore inequality \eqref{ec:ineq-bipolar} in the case of constant sectional curvature and see how can we make particular estimates on the operators to obtain remainders depending only on the curvature and distance functions. Mainly, we try and get rid of the second-order terms $\Delta_g$ and $\hess^d$.

\subsection{Bipolar $L^p-$Hardy inequalities on the hyperbolic space}
For the hyperbolic space, we consider the Poincar\'e ball model, defined as $\Hyp = \{x \in \mathbb{R}^N : |x| < 1\}$, equipped with the Riemannian metric
\begin{equation}
g_{\Hyp}(x) = (g_{ij}(x))_{i,j=\overline{1,N}} = \rho(x)^2 \delta_{ij},
\end{equation}
where $\rho(x) = \frac{2}{1 - |x|^2}$. It is well established that $(\Hyp, g_{\Hyp})$ forms a Cartan-Hadamard manifold with a constant sectional curvature of $-1$. 

The associated volume element is expressed as
\begin{equation}\label{volume-form-hyper}
    dv_{\Hyp}(x) = \rho(x)^N dx,
\end{equation}
while the hyperbolic gradient and Laplace-Beltrami operator are
\begin{equation}   \label{ec:grad-hyper}
\nabla_{\Hyp} u = \frac{\nabla u}{\rho^2}, \quad \text{and} \quad \Delta_{\Hyp} u = \rho^{-N} \dvg(\rho^{n-2} \nabla u),
\end{equation}
where $\nabla$ represents the standard Euclidean gradient in $\mathbb{R}^N$. 
The hyperbolic distance between two points $x,y \in \Hyp$ is given by
\begin{equation}
d_{\Hyp}(x,y) = \arcosh\left(1+2\frac{|x-y|^2}{(1+|x|^2)(1+|y|^2)}\right).
\end{equation}
\begin{prop}   
For any $2\leq p<N$ and $u\in C^\infty_c(\Hyp)$:
    \begin{align}
        \int_{\Hyp} \abs{\nabla_{\Hyp} u}^p dv_{\Hyp} & \geq C_1(2,p) \int_{\Hyp} \widetilde{V}_{p,a_1,a_2} \abs{u}^p dv_{\Hyp} + (N-1)C_2(2,p) \sum_{i=1}^2 \int_{\Hyp} \frac{3}{\pi^2+d_i^2} \abs{v}^{p-2}\abs{u}^p dv_{\Hyp}   \notag\\
        & + (p-2)C_2(2,p)\sum_{i=1}^2 \int_{\Hyp} \frac{1-G_{12}^2}{d_1^2d_2^2} \left[ 2+\frac{3d_i^2}{\pi^2+d_i^2} \right] \abs{v}^{p-4} \abs{u}^p dv_{\Hyp},
    \end{align}
where the constant $C_1(2,p)=(p-1)\left(\frac{N-p}{2(p-1)}\right)^p$ is sharp.
\end{prop}
\begin{proof}
    By Corollary \ref{ec:cor2}, since the sectional curvatures of the hyperbolic space are all equal to $-1$, we get the inequality along with the sharpness of the constant.
   % \begin{align}
    %    \int_{\Hyp} \abs{\nabla_{\Hyp} u}^p dv_{\Hyp} & \geq C_1(2,p) \int_{\Hyp} \widetilde{V}_{p,a_1,a_2} \abs{u}^p dv_{\Hyp} + (N-1)C_2(2,p) \sum_{i=1}^2 \int_{\Hyp} \frac{3}{\pi^2+d_i^2} \abs{v}^{p-2}\abs{u}^p dv_{\Hyp}   \notag\\
    %    & + (p-2)C_2(2,p)\sum_{i=1}^2 \int_{\Hyp} \frac{1-G^2_{12}}{d_1^2d_2^2} \left[ 2+\frac{3d_i^2}{\pi^2+d_i^2} \right] \abs{v}^{p-4} \abs{u}^p dv_{\Hyp}. 
    %\end{align} 
\end{proof}

\subsection{Bipolar $L^p-$Hardy inequalities on the upper-hemisphere}
By $\mathbb{S}^N_+$ we denote the upper-hemisphere of $\mathbb{S}^N$ endowed with the usual Riemannian metric of $\mathbb{S}^N$ inherited by $\R^{N+1}$. If we consider $a_0$ to be the north pole of $\mathbb{S}^N_+$, then we define $\delta:=\max(d(a_0,a_1),d(a_0,a_2))$. We have the following result:
\begin{prop}
For any $2\leq p<N$ and $u\in C^\infty_c(\Sp)$:
    \begin{align}
        \int_{\Sp} \abs{\nabla_g u}^p \dvol & \geq C_1(2,p)\int_{\Sp}  \abs{\frac{\nabla_g d_1}{d_1} - \frac{\nabla_g d_2}{d_2}}^2 \abs{v}^{p-2} \abs{u}^p \dvol + (N-1) C_2(2,p) c(\delta) \int_{\Sp} \abs{v}^{p-2}\abs{u}^p \dvol    \notag\\
        & + (p-2)C_2(2,p) \int_{\Sp} \left( 4 + c(\delta)(d_1^2+d_2^2) \right) \frac{1-G_{12}^2}{d_1^2d_2^2} \abs{v}^{p-4}\abs{u}^p \dvol,
    \end{align}
    where $c(\delta) = \frac{7\pi^2-3(\delta+\frac{\pi}{2})^2}{\pi^2\left((\delta+\frac{\pi}{2})^2 - \pi^2 \right)}$.
\end{prop}
\begin{proof}
    The inequality follows from Theorem \ref{ec:thm1} and certain convenient estimates arising from the influence of positive curvature of the sphere $\Sp$, rather than the case of Cartan-Hadamard manifolds. 
    By \cite[Proposition 11.3]{lee} and Theorem\eqref{ec:thmHess}, we have that 
    $$\hess^{d_i}=\frac{s_c'(d_i)}{s_c(d_i)}\pi_{d_i}, \;\;\text{for any }i=1,2.$$
    Explicitly, from \eqref{ec:proj2}, we get that
    $$\hess^{d_i}(\nabla_g d_j,\nabla_g d_j) = \cot(d_i)\left(1-G_{ij}^2\right), \;\;\text{for any } i,j=1,2,\;\; i\neq j.$$
    Consequently, 
    $$\Delta d_i = (N-1)\cot(d_i), \;\; \text{for any }i=1,2.$$
    We will use the Mittag-Leffler expansion of the cotangent function (see, e.g., \cite{ahlfors}), that is 
    \begin{equation}  \label{ec:mitlef}
        \cot(t) = \frac{1}{t} + 2t\sum_{k=1}^\infty \frac{1}{t^2-\pi^2k^2}, \;\; t\in(0,\pi).
    \end{equation}
    Taking into account that $d_i\in (0,\pi)$ in $\Sp$ and that $\delta<\frac{\pi}{2}$, we get that $d_i(x)\leq d(a_0,a_i)+d(a_0,x) < \delta + \frac{\pi}{2}$. Hence, 
    \begin{align}
        \sum_{i=1}^2 \int_{\Sp} & \frac{d_i\Delta d_i - (N-1)}{d_i^2} \abs{v}^{p-2}\abs{u}^p \dvol = (N-1)\sum_{i=1}^2 \int_{\Sp} \frac{d_i\cot(d_i) - 1}{d_i^2} \abs{v}^{p-2}\abs{u}^p \dvol   \notag\\
        & = 2(N-1)\sum_{i=1}^2 \int_{\Sp} \sum_{k=1}^\infty\frac{1}{d_i^2 - \pi^2k^2} \abs{v}^{p-2}\abs{u}^p \dvol    \notag\\
        & \geq 2(N-1) \sum_{i=1}^2 \left(\frac{1}{(\delta+\frac{\pi}{2})^2 - \pi^2} + \frac{1}{\pi^2}\sum_{k=2}^\infty\frac{1}{1 - k^2} \right) \int_{\Sp} \abs{v}^{p-2}\abs{u}^p \dvol    \notag\\
        & = 4(N-1) \left(\frac{1}{(\delta+\frac{\pi}{2})^2 - \pi^2} - \frac{3}{4\pi^2}\right) \int_{\Sp} \abs{v}^{p-2}\abs{u}^p \dvol    \notag\\
        & = (N-1) \frac{7\pi^2-3(\delta+\frac{\pi}{2})^2}{\pi^2\left((\delta+\frac{\pi}{2})^2 - \pi^2 \right)} \int_{\Sp} \abs{v}^{p-2}\abs{u}^p \dvol.    \notag
    \end{align}
    Next, we turn our attention to the Hessian term in the inequality. In the same manner, we get that
    \begin{align}
        \sum_{\substack{i,j=1\\ i\neq j}}^2 \int_{\Sp} & \frac{\hess^{d_i} (\nabla_g d_j,\nabla_g d_j)}{d_id_j^2} \abs{v}^{p-4}\abs{u}^p \dvol = \sum_{i=1}^2 \int_{\Sp} d_i\cot(d_i)\frac{1-G_{12}^2}{d_1^2d_2^2} \abs{v}^{p-4}\abs{u}^p \dvol    \notag\\
        & = \sum_{i=1}^2 \int_{\Sp} \left( 1+2d_i^2\sum_{k=1}^\infty\frac{1}{d_i^2-\pi^2k^2} \right) \frac{1-G_{12}^2}{d_1^2d_2^2} \abs{v}^{p-4}\abs{u}^p \dvol    \notag\\
        & \geq \sum_{i=1}^2 \int_{\Sp} \left( 1+c(\delta)d_i^2 \right) \frac{1-G_{12}^2}{d_1^2d_2^2} \abs{v}^{p-4}\abs{u}^p \dvol    \notag\\
        & = \int_{\Sp} \left( 2+c(\delta)(d_1^2+d_2^2) \right) \frac{1-G_{12}^2}{d_1^2d_2^2} \abs{v}^{p-4}\abs{u}^p \dvol.    \notag
    \end{align}
From \eqref{ec:ineq-bipolar}, 
    \begin{align}
        \int_{\Sp} \abs{\nabla_g u}^p \dvol & \geq C_1(2,p)\int_{\Sp}  \abs{\frac{\nabla_g d_1}{d_1} - \frac{\nabla_g d_2}{d_2}}^2 \abs{v}^{p-2} \abs{u}^p \dvol + (N-1) C_2(2,p) c(\delta) \int_{\Sp} \abs{v}^{p-2}\abs{u}^p \dvol    \notag\\
        & + (p-2)C_2(2,p) \int_{\Sp} \left( 4 + c(\delta)(d_1^2+d_2^2) \right) \frac{1-G_{12}^2}{d_1^2d_2^2} \abs{v}^{p-4}\abs{u}^p \dvol,
    \end{align}
which concludes the proof.    
\end{proof}

\end{document}